\documentclass[12pt,reqno]{article}

\usepackage[usenames]{color}
\usepackage{amssymb}
\usepackage{graphicx}
\usepackage{amscd}

\usepackage[colorlinks=true,
linkcolor=webgreen,
filecolor=webbrown,
citecolor=webgreen]{hyperref}

\definecolor{webgreen}{rgb}{0,.5,0}
\definecolor{webbrown}{rgb}{.6,0,0}

\usepackage{color}
\usepackage{fullpage}
\usepackage{float}

\usepackage{graphics,amsmath,amssymb}
\usepackage{amsthm}
\usepackage{amsfonts}
\usepackage{latexsym}
\usepackage{epsf}

\usepackage[english]{babel}
\usepackage{rotating}

\newcommand{\seqnum}[1]{\href{http://oeis.org/#1}{\underline{#1}}}
\def\modd#1 #2{#1\ ({\rm mod}\ #2)}

\begin{document}


\theoremstyle{plain}
\newtheorem{theorem}{Theorem}
\newtheorem{corollary}[theorem]{Corollary}
\newtheorem{lemma}[theorem]{Lemma}

\theoremstyle{definition}
\newtheorem*{example}{Example}

\begin{center}
\vskip 1cm{\LARGE\bf Early Pruning in the \\
\vskip .1in
Restricted Postage Stamp Problem
}
\vskip 1cm
\large
Jukka Kohonen \\
Department of Mathematics and Statistics\\
P. O. Box 68\\
FI-00014 University of Helsinki\\
Finland \\
\href{mailto:jukka.kohonen@helsinki.fi}{\tt jukka.kohonen@helsinki.fi}
\end{center}

\vskip .2 in

\begin{abstract}
  A set of non-negative integers is an additive basis with range $n$,
  if its sumset covers all consecutive integers from $0$ to $n$, but
  not $n+1$.  If the range is exactly twice the largest element of the
  basis, the basis is restricted.  Restricted bases have important
  special properties that facilitate efficient searching.  With the
  help of these properties, we have previously listed the extremal
  restricted bases up to length $k=41$.  Here, with a more prudent use
  of the properties, we present an improved search algorithm and list
  all extremal restricted bases up to $k=47$.
\end{abstract}

\section{Introduction}
Let
$$A = \{a_0 < a_1 < \cdots < a_k\}$$
 be a set of $k+1$ non-negative integers,
and 
$$
2A := \{a + a' : a,a' \in A\}
$$ its {\it sumset}.  If $2A$ contains the consecutive integers $[0,n]
:= \{0,1,\ldots,n\}$, but $n+1 \notin 2A$, then $A$ is an (additive)
{\it basis} of length $k$ and range $n_2(A) = n$.  Note that the
smallest element must be $a_0 = 0$ (otherwise the sumset would not
contain $0$).

An additive basis $A$ is {\em admissible} if $n_2(A) \ge a_k$, and
{\em restricted} if $n_2(A) = 2a_k$.  Restricted bases are admissible
by definition.  Also, $A$ is restricted if and only if $2A =
[0,2a_k]$.

\begin{example} If $A=\{0,1,3,4\}$, then $2A=[0,8]$, and $A$ is
a restricted basis with range $n_2(A) = 8 = 2a_k$.
\end{example}

\begin{example} If $A=\{0,1,2,4\}$, then $2A = [0,6] \cup \{8\}$,
and $A$ is an admissible (but not restricted) basis with range $n_2(A)
= 6 < 2a_k$.
\end{example}

The maximum range among {\em all} bases of length $k$ is denoted by
$n_2(k)$, and the maximum among {\em restricted} bases is $n_2^*(k)$.
The bases that attain these maxima are called {\it extremal bases} and
{\it extremal restricted bases}, respectively
\cite{riddell1978,wagstaff1979}.  Searching for extremal bases is
known in the literature as the {\em postage stamp problem}.  Searching
for extremal restricted bases could then be called the {\em restricted
  postage stamp problem}.

Restricted bases have important properties that facilitate efficient
searching: mirroring and lower bounds.  Using them, we have previously
presented a ``meet-in-the-middle'' algorithm, and enumerated all
extremal restricted bases up to length $k=41$
\cite{kohonen2014b,sloane}.  Here we improve the algorithm by a more
careful use of the properties, and enumerate all extremal restricted
bases up to $k=47$.

\section{Properties of restricted bases}

Let us revisit some properties of restricted bases
\cite{kohonen2014b}.  The mirroring property
\cite[Theorem~5]{kohonen2014b} is based on a reasoning similar to
Rohrbach's theorem for symmetric bases \cite[Satz~1]{rohrbach1937},
but holds for asymmetric restricted bases as well.

\begin{theorem}[Mirroring] If $A$ is a restricted basis with
range~$n$, then its {\em mirror image}
$$B = a_k - A = \{a_k - a : a \in A\}$$ is also a restricted basis
with the same range.
\label{thm:mirror}
\end{theorem}
\begin{proof}
\begin{align*}
2B &= \{b+b' : b,b' \in B\} = \{(a_k-a)+(a_k-a') : a,a' \in A\} \\
   &= 2a_k - 2A = n - [0,n] = [0,n]. \qedhere
\end{align*}
\end{proof}

\begin{example} Let $A=\{0,1,2,3,7,11,15,17,20,21,22\}$.  This is
a restricted basis with range~$44$.  Its mirror image $B = 22-A =
\{0,1,2,5,7,11,15,19,20,21,22\}$ is another restricted basis with the
same range.
\end{example}

If $A_k = \{a_0 < a_1 < \cdots < a_k\}$, we define its $j$-{\em
  prefix} as $A_j = \{a_0, \ldots, a_j\}$, for any $0 \le j \le k$.
The following upper bounds hold for all admissible bases (including
all restricted bases).  For restricted bases, the upper bounds can be
mirrored to obtain lower bounds as well.

\begin{lemma}
  If $A_k$ is an admissible basis, and $1 \le j \le k$, then $a_j \le
  n_2(A_{j-1})+1$.
  \label{thm:admissible}
\end{lemma}
\begin{proof}
  Represent $A_k$ as a disjoint union $A_k = A_{j-1} \cup R$, where $r
  \ge a_j$ for all $r \in R$.  Now $2A_k = (2A_{j-1}) \cup (R+A_k)$.
  All elements of $(R+A_k)$ are greater or equal to $a_j$, thus
  $2A_{j-1}$ must cover the interval $[0,a_j-1]$.  In other words
  $n_2(A_{j-1}) \ge a_j-1$.
\end{proof}

\begin{theorem}[Element-wise upper bound]
  If $A_k$ is an admissible basis, and $1 \le j \le k$, then $a_j \le
  n_2(j-1)+1$.
\label{thm:upper}
\end{theorem}
\begin{proof}
  Follows from Lemma~\ref{thm:admissible} because $n_2(A_{j-1}) \le
  n_2(j-1)$.
\end{proof}

\begin{theorem}[Element-wise lower bound]
  If $A_k$ is a restricted basis, and $0 \le j \le
  k-1$, then $a_j \ge a_k - n_2(k-j-1) - 1$.
  \label{thm:lower}
\end{theorem}
\begin{proof}
  Let $B_k = a_k - A_k$.  By Theorem~\ref{thm:mirror}, $B_k$ is a
  restricted basis, and thus admissible.  Let $i=k-j$.  By
  Theorem~\ref{thm:upper} we have $b_i \le n_2(i-1) + 1$, thus
  \begin{equation*}
    a_j = a_k - b_i \ge a_k - n_2(k-j-1) - 1. \qedhere
  \end{equation*}
\end{proof}

\begin{corollary}[Range lower bound]
  If $A_k$ is a restricted basis, and $0 \le j \le k-2$, then
  $n_2(A_j) \ge a_k - n_2(k-j-2) - 2$.
  \label{thm:lowerrange}
\end{corollary}
\begin{proof}
  Follows from the previous theorem since $a_{j+1} \le n_2(A_j)+1$.
\end{proof}

\section{Searching for restricted bases}
\label{sec:search}

\begin{figure}[bt]
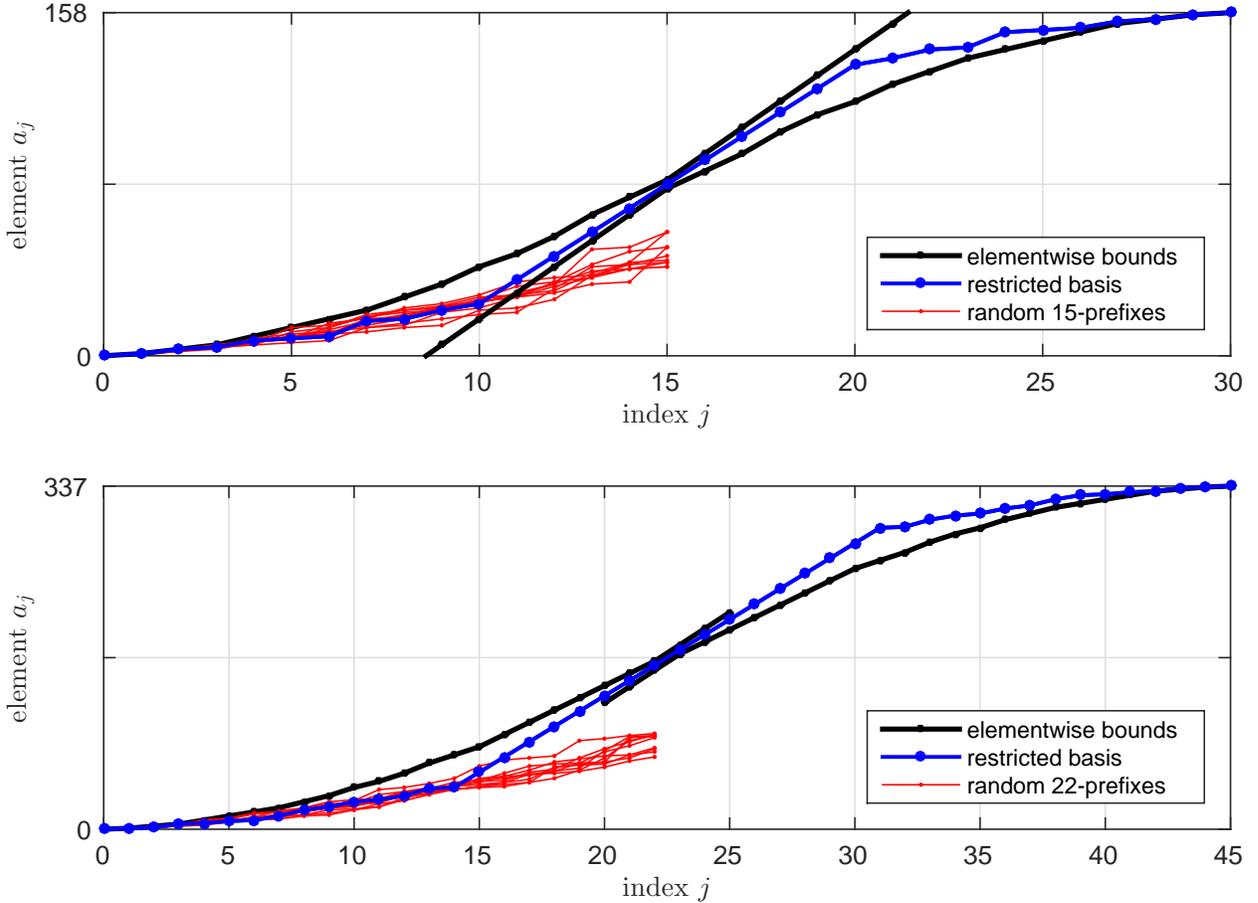

  \includegraphics[width=\textwidth]{figmime30}

  \vspace{0.5cm}
  \includegraphics[width=\textwidth]{figmime45}

  \caption{Element-wise bounds for restricted bases.  Top: $k=30$ and
    $n=316$.  Bottom: $k=45$ and $n=674$.  Thick blue line: a
    restricted basis.  Thin red lines: ten randomly generated
    admissible prefixes.}
  \label{fig:bounds}
\end{figure}

The bounds are easily calculated if the corresponding $n_2$ is known
(sequence \seqnum{A001212} in Sloane's OEIS \cite{sloane}).  The
element-wise bounds are quite narrow near the middle of a basis, as
seen in Figure~\ref{fig:bounds}.  In the vast majority of admissible
prefixes, the middle elements are far below the lower bound
(illustrated with random admissible prefixes in the figure).

\begin{example}
  Search for a restricted basis of length $k=30$ and range $n=316$
  (thus $a_k = n/2 = 158$).  From Theorem~\ref{thm:lower} we have
  $a_{15} \ge 77$.  While there are $9\;041\;908\;204$ admissible
  $15$-prefixes (\seqnum{A167809}), only $201$ of them meet the lower
  bound for $a_{15}$, and are possible prefixes for the restricted
  basis.
\end{example}

Alternatively, we could use the {\em range} bound at midpoint
($j=\lfloor k/2 \rfloor$): from Corollary~\ref{thm:lowerrange} we
obtain $n_2(A_{15}) \ge 84$.  Our previously presented algorithm
\cite[Algorithm~1]{kohonen2014b} was built upon this idea.  Challis's
algorithm \cite{challis1993} was used to enumerate the admissible
$j$-prefixes that meet the range bound.

However, if prefixes are being built progressively (adding one element
at a time), many proposed prefixes can be rejected much {\em before}
the midpoint (see Figure~\ref{fig:bounds}, top).  It is
straightforward to modify Challis's algorithm to check for the lower
bounds at each element, and to reject a prefix as soon as any element
violates the lower bound.  This approach prunes the search tree and
speeds up the search tremendously.

\begin{example}
  Searching for a restricted basis with $k=30$ and $n=316$,
  Algorithm~1 uses only the range bound $n_2(A_{15}) \ge 84$.  During
  the search it visits about $4.0 \times 10^8$ prefixes, taking about
  30 CPU seconds on our system.  It generates $791$ possible
  $15$-prefixes.

  For elements $a_{10},a_{11},\ldots,a_{15}$ we have the lower bounds
  $17,29,41,53,65,$ and $77$, respectively.  The modified search,
  which exploits these bounds, visits only about $1.9 \times 10^6$
  prefixes ($200$ times fewer than Algorithm~1), runs in about $0.1$
  CPU seconds, and generates only $16$ possible $15$-prefixes.
\end{example}

With large values of $k$, a further complication is that $n_2$ is
known only up to length $24$ \cite{kohonen2014}.  For example, if
$k=45$, the element-wise lower bounds are known for $j \ge 20$ (see
Figure~\ref{fig:bounds}, bottom).  In order to use
Theorem~\ref{thm:lower} for $j=19$, we would need $n_2(k-19-1) =
n_2(25)$, which is not known.  This is a serious limitation: in the
search for possible prefixes, the known element-wise bounds kick in at
$j=20$.  If the bounds were known, it seems plausible that most
prefixes could be rejected earlier, perhaps around $j=17$.

What we can do, with large $k$, is to use the range bound as early as
possible.  For $k=45$, $n=674$, Corollary~\ref{thm:lowerrange} gives
the bound $n_2(A_{19}) \ge 123$.  Using this as the target range in
Challis's algorithm, we can first enumerate the possible $19$-prefixes
and then extend them by continuing the algorithm (checking for
element-wise bounds at every step).  With the range bound, the
so-called {\em gaps test} in Challis's algorithm rejects many prefixes
even before $j=19$.

\section{Results}

With the method described in the previous section, we computed all
extremal restricted bases of lengths $k=42,\ldots,47$.  The prefix
computations are illustrated in Table~\ref{table:time}.  Extending the
prefixes and joining them with suffixes (as in our previous algorithm
\cite[Algorithm~1]{kohonen2014b}) into complete bases was then a
matter of a few seconds or minutes at most.  Since $n_2^*$ is {\em a
  priori} unknown, we started with the range $n$ set to its upper
bound \cite[Corollary~8]{kohonen2014b} and decreased in steps of 2,
until a restricted basis was found.

Previously, with Algorithm~1, we used 120 CPU hours to find extremal
restricted bases for $k=41$, which illustrates the strong effect of
using the early lower bounds for pruning.

\begin{table}[htb]
  \begin{center}
    \begin{tabular}{c|c|l|l|c|r}
      $k$ & $n$ & range bound & work & CPU hours & prefixes generated \\
      \hline
      42 & 588 & $n_2(A_{16})\ge 80$  & $9.6\times 10^{9}$  & $0.7$ & $28\;026\;041$ \\
      43 & 614 & $n_2(A_{17})\ge 93$  & $7.2\times 10^{10}$ & $2.0$ &  $4\;375\;029$ \\
      44 & 644 & $n_2(A_{18})\ge 108$ & $3.8\times 10^{11}$ & $8.9$ &     $317\;752$ \\
      45 & 674 & $n_2(A_{19})\ge 123$ & $1.5\times 10^{12}$ &  $35$ &      $44\;187$ \\
      46 & 704 & $n_2(A_{20})\ge 138$ & $6.4\times 10^{12}$ & $157$ &      $11\;448$ \\
      47 & 734 & $n_2(A_{21})\ge 153$ & $3.2\times 10^{13}$ & $812$ &       $4\;020$
    \end{tabular}
  \end{center}
  \caption{Computing possible prefixes for restricted bases of lengths
    $k=42,\ldots,47$.  {\em Range bound} is from
    Corollary~\ref{thm:lowerrange}, with $j$ as small as possible.
    {\em Work} is the number of prefixes visited during the search.
    {\em Prefixes generated} is the number of prefixes that meet the
    range bound.}
  \label{table:time}
\end{table}

The complete bases are listed in Table~\ref{table:bases}.  They are
all symmetric (that is, $A_k = a_k - A_k$), which was not known nor
enforced {\em a~priori}.  The bases are exactly those proposed by
Challis and Robinson's preamble-amble construction \cite[Table
  2]{challis2010}.  The result of our computation here is that (1)
these are indeed {\em extremal} restricted bases, and that (2) this is
the {\em complete} listing of extremal restricted bases of these
lengths.

\section{Discussion}

As mentioned in Section~\ref{sec:search}, efficient searching for {\em
  restricted} additive bases with our method depends crucially on the
availability of element-wise lower bounds, which in turn depends on the
knowledge of extremal {\em unrestricted} ranges $n_2$
(\seqnum{A001212}).  Roughly speaking, if $n_2$ is known up to length
$k$ (currently 24), then it provides lower bounds that are useful
for computing of $n_2^*$ up to about length $2k$.

To extend our knowledge of extremal restricted bases further, an
obvious way would be to compute first the unrestricted $n_2(k)$ for
greater lengths, say, $k=25$, and use them to provide improved lower
bounds for the restricted case.

A more interesting question is, can any connection be established
between $n_2(k)$ and $n_2^*(k)$ (\seqnum{A001212} and
\seqnum{A006638})?  For example, can it be shown that $n_2(k)-n_2^*(k)
\le d$ with some small value $d$?  For lengths $k \le 24$, where both
quantities are currently known, the difference is always zero or two
(the latter only with $k=10$, where $n_2(10)=46$ and $n_2^*(10)=44$).
If the difference could be bounded to be small, then $n_2^*(k)+d$
could be used as an upper bound for $n_2(k)$, providing in turn the
lower bounds for computing $n_2^*$ for greater lengths.

\bigskip
\hrule
\bigskip
\noindent 2000 {\it Mathematics Subject Classification}:
Primary 11B13.

\noindent \emph{Keywords: } additive basis, restricted basis.

\bigskip
\hrule
\bigskip

\noindent (Concerned with sequences \seqnum{A001212},
\seqnum{A006638}, and \seqnum{A167809}.)

\bigskip
\hrule
\bigskip

\begin{sidewaystable}[p]
\begin{center}
\small \setlength{\tabcolsep}{1.7pt}
\begin{tabular}{cc|rrrrrrrrrrrrrrrrrrrrrrrrrrrrrrrrr}
$k$ & $n_2^*(k)$ & \multicolumn{5}{l}{basis} \\
\hline
 42 & 588 &  0&   1&   2&   5&   7&  10&  11&  19&  21&  22&  25&  29&  30& $\cdots$ & +13 & $\cdots$& 264& 265& 269& 272& 273& 275& 283& 284& 287& 289& 292& 293& 294\\
\hline
 43 & 614 &  0&   1&   2&   5&   7&  10&  11&  19&  21&  22&  25&  29&  30& $\cdots$ & +13 & $\cdots$& 277& 278& 282& 285& 286& 288& 296& 297& 300& 302& 305& 306& 307\\
 43 & 614 &  0&   1&   2&   5&   6&   8&   9&  13&  19&  22&  27&  29&  33&  40&  41& $\cdots$ & +15 & $\cdots$& 266& 267& 274& 278& 280& 285& 288& 294& 298& 299& 301& 302& 305& 306& 307\\
\hline
 44 & 644 &  0&   1&   2&   5&   6&   8&   9&  13&  19&  22&  27&  29&  33&  40&  41& $\cdots$ & +15 & $\cdots$& 281& 282& 289& 293& 295& 300& 303& 309& 313& 314& 316& 317& 320& 321& 322\\
\hline
 45 & 674 &  0&   1&   2&   5&   6&   8&   9&  13&  19&  22&  27&  29&  33&  40&  41& $\cdots$ & +15 & $\cdots$& 296& 297& 304& 308& 310& 315& 318& 324& 328& 329& 331& 332& 335& 336& 337\\
\hline
 46 & 704 &  0&   1&   2&   5&   6&   8&   9&  13&  19&  22&  27&  29&  33&  40&  41& $\cdots$ & +15 & $\cdots$& 311& 312& 319& 323& 325& 330& 333& 339& 343& 344& 346& 347& 350& 351& 352\\
\hline
 47 & 734 &  0&   1&   2&   5&   6&   8&   9&  13&  19&  22&  27&  29&  33&  40&  41& $\cdots$ & +15 & $\cdots$& 326& 327& 334& 338& 340& 345& 348& 354& 358& 359& 361& 362& 365& 366& 367\\
\hline

\end{tabular}
\caption{Extremal restricted bases of lengths $k=42,\ldots,47$.  The
  notation $+c$ indicates several elements with a repeated difference
  of $c$.}
\label{table:bases}
\end{center}
\end{sidewaystable}


\begin{thebibliography}{9}

\bibitem{challis1993} M. F. Challis, Two new techniques for computing
  extremal $h$-bases $A_k$, {\em The Computer Journal} {\bf 36}
  (1993), 117--126.

\bibitem{challis2010} M. F. Challis and J. P. Robinson, Some extremal
  postage stamp bases, {\em J. Integer Seq.} {\bf 13} (2010), Article
  10.2.3.

\bibitem{kohonen2014b} J. Kohonen, A meet-in-the-middle algorithm for
  finding extremal restricted additive 2-bases, {\em J. Integer
    Seq.} {\bf 17} (2014), Article 14.6.8.

\bibitem{kohonen2014} J. Kohonen and J. Corander, Addition chains meet
  postage stamps: Reducing the number of multiplications, {\em
    J. Integer Seq.} {\bf 17} (2014), Article 14.3.4.

\bibitem{riddell1978} J. Riddell and C. Chan, Some extremal 2-bases,
  {\em Math. Comp.} {\bf 32} (1978), 630--634.

\bibitem{rohrbach1937} H. Rohrbach, Ein Beitrag zur additiven
  Zahlentheorie, {\em Math. Z.} {\bf 42} (1937), 1--30.

\bibitem{sloane} N. J. A. Sloane, On-Line Encyclopedia of Integer Sequences,
\url{http://oeis.org}.

\bibitem{wagstaff1979} S. S. Wagstaff, Additive $h$-bases for $n$, in
  M. B. Nathanson, ed., {\em Number Theory Carbondale 1979}, Lect.
  Notes in Math., Vol.~751, Springer, 1979, pp.~302--327.

\end{thebibliography}
\end{document}